\documentclass[a4paper,11pt]{amsart}
\usepackage{fancyhdr}
\usepackage{amsmath}
\usepackage{dsfont}
\usepackage{hyperref}
\usepackage[mathscr]{eucal}
\usepackage[cp1251]{inputenc}
\usepackage[english]{babel}
\usepackage{cite,enumerate,float,indentfirst}
\usepackage{graphicx}
\usepackage{xcolor}
\usepackage{latexsym,a4,mathrsfs,amsthm,amsmath,amssymb,url}
\usepackage{amsfonts}
\usepackage{amssymb}
\usepackage{epstopdf}

\numberwithin{equation}{section}
\newtheorem{theorem}{Theorem}[section]
\newtheorem{lemma}[theorem]{Lemma}

\theoremstyle{remark}
\newtheorem*{remark}{Remark}
\newtheorem*{example}{Example}

\newcommand{\R}{\mathbb{R}}

\newcommand{\X}{{\bf X}}
\newcommand{\A}{{\bf A}}
\newcommand{\D}{{\bf D}}

\newcommand{\B}{{\bf B}}

\newcommand{\I}{{\bf I}}
\newcommand{\RR}{{\bf R}}

\newcommand{\W}{{\bf W}}

\newcommand{\Z}{{\bf Z}}
\newcommand{\Y}{{\bf Y}}

\newcommand{\ee}{{\bf e}}

\newcommand{\vect}[1]{\boldsymbol{#1}}

\DeclareMathOperator{\Tr}{Tr}

\DeclareMathOperator{\E}{\mathbb{E}}
\DeclareMathOperator{\Pb}{\mathbb{P}}
\DeclareMathOperator{\one}{\mathds{1}}

\begin{document}

\vspace{1in}

\title[SEMICIRCLE LAW]{\bf SEMICIRCLE LAW FOR A CLASS OF RANDOM MATRICES WITH DEPENDENT ENTRIES}


\author[F. G{\"o}tze]{F. G{\"o}tze}
\address{F. G{\"o}tze\\
 Faculty of Mathematics\\
 Bielefeld University \\
 Bielefeld, Germany
}
\email{goetze@math.uni-bielefeld.de}

\author[A. Naumov]{A. Naumov}
\address{A. Naumov\\
 Faculty of Mathematics\\
 Bielefeld University \\
 Bielefeld, Germany \\
 and Faculty of Computational Mathematics and Cybernetics\\
 Moscow State University \\
 Moscow, Russia
 }
\email{naumovne@gmail.com, anaumov@math.uni-bielefeld.de}

\author[A. Tikhomirov]{A. Tikhomirov}
\address{A. Tikhomirov\\
 Department of Mathematics\\
 Komi Research Center of Ural Branch of RAS \\
 Syktyvkar, Russia
 }
\email{tichomir@math.uni-bielefeld.de}
\thanks{All authors are supported by CRC 701 ``Spectral Structures and Topological
Methods in Mathematics'', Bielefeld. A.~Tikhomirov are partially supported by RFBR, grant  N 11-01-00310-a and grant N 11-01-12104-ofi-m-2011, and program of UD RAS No12-P-1-1013. A.~Naumov has been supported by the German Research Foundation (DFG) through the International Research Training Group IRTG 1132}


\date{\today}

\begin{abstract}
In this paper we study ensembles of random symmetric matrices $\X_n = \{X_{ij}\}_{i,j = 1}^n$ with a random field type dependence, such that $\E X_{ij} = 0$, $\E X_{ij}^2 = \sigma_{ij}^2$, where $\sigma_{ij}$
can be different numbers. Assuming that the average of the normalized sums of variances in each row converges to one and Lindeberg condition holds true we prove that the empirical spectral distribution of eigenvalues converges to
Wigner's semicircle  law.
\end{abstract}

\maketitle
\tableofcontents


\newpage

\section{Introduction}

Let $X_{jk}, 1 \le j \le k < \infty$, be triangular array of random variables with $\E X_{j k} = 0$ and $\E X_{jk}^2 = \sigma_{j k}^2$, and let $X_{jk} = X_{kj}$ for $1 \le j < k < \infty$. We consider the random matrix
$$
\X_n = \{ X_{jk} \}_{j,k=1}^n.
$$
Denote by $\lambda_1 \le ... \le \lambda_n$ eigenvalues of matrix $n^{-1/2} \X_n$ and define its spectral distribution function by
$$
\mathcal F^{\X_n}(x) = \frac{1}{n} \sum_{i = 1}^n \one(\lambda_i \le x),
$$
where $\one(B)$ denotes the indicator of an event $B$. We set $F^{\X_n}(x):=\E \mathcal F^{\X_n}(x)$.
Let $g(x)$ and $G(x)$ denote the density and the distribution function of the standard semicircle  law
\begin{equation*}
g(x) = \frac{1}{2\pi} \sqrt{4 - x^2} \one(|x| \le 2), \quad G(x) = \int_{-\infty}^x g(u) du.
\end{equation*}
For matrices with independent identically distributed (i.i.d.) elements, which have moments of all orders, Wigner proved in~\cite{Wigner1958} that $F_n$ converges to $G(x)$, later on called ``Wigner's semicircle  law``. The result has been extended in various aspects, i.e. by Arnold in~\cite{Arnold1971}. In the non-i.i.d. case Pastur,~\cite{Pastur1973}, showed that Lindeberg's condition is sufficient for the  convergence. In~\cite{GotTikh2006} G{\"o}tze and Tikhomirov proved the semicircle law for matrices satisfying martingale-type conditions for the entries.

In the majority of previous papers it has been assumed that $\sigma_{ij}^2$ are equal for all $1 \le i < j \le n$.
Recently Erd{\H o}s, Yau and Yin and al. study ensembles of symmetric random matrices with independent elements which satisfy
$n^{-1}\sum_{j=1}^n \sigma_{ij}^2 = 1$ for all $1 \le i \le n$. See for example the survey of results in~\cite{Erdos2010}.

In this paper we study a class of random matrices with dependent entries and show that limiting distribution for $F^{\X_n}(x)$ is given by Wigner's semicircle  law. We do not assume that the variances are equal.

Introduce the $\sigma$-algebras
$$
\mathfrak{F}^{(i,j)} := \sigma\{X_{k l}: 1 \le k \le l \le n,  (k,l) \neq (i,j) \}, 1 \le i \le j \le n.
$$
For any $\tau > 0$ we introduce Lindeberg's ratio for random matrices as
\begin{equation*} \label{eq: LF}
L_n(\tau):=\frac{1}{n^2} \sum_{i,j = 1}^n \E |X_{ij}|^2 \one(|X_{ij}| \geq \tau \sqrt n).
\end{equation*}
We assume that the following conditions hold
\begin{align}
&\E (X_{i j} | \mathfrak{F}^{(i,j)}) = 0; \label{eq:mrtingale prop}\\
&\frac{1}{n^2} \sum_{i,j = 1}^n \E|\E(X_{ij}^2|\mathfrak{F}^{(i,j)}) - \sigma_{ij}^2| \rightarrow 0 \text{ as } n \rightarrow \infty; \label{eq: variance prop}\\
&\text{for any fixed } \tau > 0 \quad L_n(\tau) \rightarrow 0 \text{ as } n \rightarrow \infty. \label{eq: limit of Lind. ratio}
\end{align}
Furthermore, we will use condition~\eqref{eq: limit of Lind. ratio} not only for the matrix $\X_n$, but for other matrices as well, replacing $X_{ij}$ in the definition of Lindeberg's ratio by corresponding elements.

For all $1 \le i \le n$ let $B_i^2 := \frac{1}{n} \sum_{j=1}^n \sigma_{ij}^2$. We need to impose additional conditions on the variances $\sigma_{ij}^2$ given by
\begin{align}\label{eq: variance prop 2}
&\frac{1}{n}\sum_{i=1}^n |B_i^2 - 1| \rightarrow 0 \text{ as } n \rightarrow \infty;\\
&\max_{ 1 \le i \le n} B_i \le C \label{eq: variance prop 3},
\end{align}
where $C$ is some absolute constant.
\begin{remark}
It is easy to see that the conditions~\eqref{eq: variance prop 2} and~\eqref{eq: variance prop 3} follow from the following condition
\begin{equation} \label{eq: variance prop 2b}
\max_{1 \le i \le n} \left|B_i^2 - 1 \right| \rightarrow 0 \text{ as } n \rightarrow \infty.
\end{equation}
\end{remark}
The main result of the paper is the following theorem
\begin{theorem} \label{th:main}
Let $\X_n$ satisfy conditions~\eqref{eq:mrtingale prop}--\eqref{eq: variance prop 3}.
Then
$$
\sup_{x}|F^{\X_n}(x) - G(x)| \rightarrow 0 \text{ as } n \rightarrow \infty.
$$
\end{theorem}
Let us fix $i, j$. It is easy to see that for all $(k,l) \neq (i,j)$
$$
\E X_{i j} X_{k l} = \E \E (X_{i j} X_{k l} |\mathfrak{F}^{(i,j)})) = \E X_{k l} \E  (X_{i j} |\mathfrak{F}^{(i,j)}) = 0.
$$
Hence the elements of the matrix $\X_n$ are uncorrelated. If we additionally assume that the elements of the matrix $\X_n$ are independent random variables then conditions~\eqref{eq:mrtingale prop} and~\eqref{eq: variance prop} are automatically satisfied. The following
theorem follows immediately from Theorem~\ref{th:main} in the case when the matrix $\X_n$ has independent entries.
\begin{theorem} \label{th:indep case}
Assume that the elements $X_{ij}$ of the matrix $\X_n$ are independent for all $1 \le i \le j \le n$ and $\E X_{ij} = 0$, $\E X_{ij}^2 = \sigma_{ij}^2$. Assume that $\X_n$ satisfies
conditions~\eqref{eq: limit of Lind. ratio}--\eqref{eq: variance prop 3}.
Then
$$
\sup_{x}|F^{\X_n}(x) - G(x)| \rightarrow 0 \text{ as } n \rightarrow \infty.
$$
\end{theorem}
The following example illustrates that without condition~\eqref{eq: variance prop 2} convergence to  Wigner's semicircle  law doesn't hold.
\begin{example}
Let $\X_n$ denote a block matrix
\begin{equation*}
\X_n = \begin{pmatrix}
\A  & \B\\
\B^T  & \D\\
\end{pmatrix},
\end{equation*}
where $\A$ is $m \times m$ symmetric random matrix with Gaussian elements with zero mean and unit variance, $\B$ is $m \times (n-m)$ random matrix with i.i.d. Gaussian elements with zero mean and unit variance. Furthermore, let $\D$ be a $(n-m)\times (n-m)$ diagonal matrix with Gaussian random variables on the diagonal with zero mean and unit variance. If we set $m:=n/2$ then it is not difficult to check that condition~\eqref{eq: variance prop 2} doesn't hold. We simulated  the spectrum of the matrix $\X_n$ and illustrated a limiting distribution on Figure~\ref{fig:graph2}.
\begin{figure}
\begin{center}
\scalebox{.35}{\includegraphics{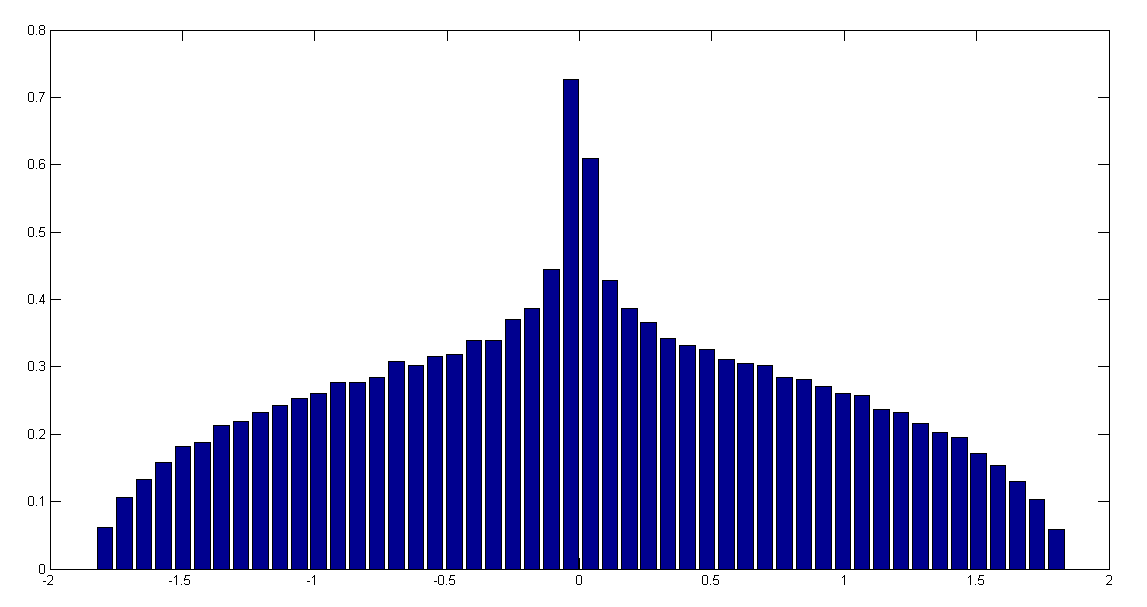}}
\end{center}
\caption{Spectrum of matrix $\X_n$.}
\label{fig:graph2}
\end{figure}
\end{example}
\begin{remark}
We conjecture that Theorem~\ref{th:main} (Theorem~\ref{th:indep case} respectively) holds without assumption~\eqref{eq: variance prop 3}.
\end{remark}
Define the Levy distance between the distribution functions $F_1$ and $F_2$ by
$$
L(F_1, F_2)= \inf\{\varepsilon > 0: F_1(x-\varepsilon) - \varepsilon \le F_2(x) \le F_1(x+\varepsilon) + \varepsilon \}.
$$
The following theorem formulates the Lindeberg's universality scheme for random matrices.
\begin{theorem} \label{th:universality}
Let $\X_n, \Y_n$ denote independent symmetric random matrices with $\E X_{ij} = \E Y_{ij} = 0$ and $\E X_{ij}^2 = \E Y_{ij}^2 = \sigma_{ij}^2$. Suppose that
the matrix $\X_n$ satisfies conditions~\eqref{eq:mrtingale prop}--\eqref{eq: variance prop 2}, and the matrix $\Y_n$ has independent Gaussian elements. Additionally assume that for the matrix $\Y_n$
conditions~\eqref{eq: limit of Lind. ratio}--\eqref{eq: variance prop 2} hold. Then
$$
L(F^{\X_n}(x), F^{\Y_n}(x)) \rightarrow 0 \text{ as } n \rightarrow \infty.
$$
\end{theorem}
In view of Theorem~\ref{th:universality} to prove Theorem~\ref{th:main} it remains to show convergence to semicircle  law in the Gaussian case.
\begin{theorem} \label{th: gaussian case}
Assume that the entries $Y_{ij}$ of the matrix $\Y_n$ are independent for all $1 \le i \le j \le n$ and have Gaussian distribution with $\E Y_{ij} = 0$, $\E Y_{ij}^2 = \sigma_{ij}^2$.
Assume that conditions~\eqref{eq: limit of Lind. ratio}--\eqref{eq: variance prop 3} are satisfied.
Then
$$
\sup_{x}|F^{\Y_n}(x) - G(x)| \rightarrow 0 \text{ as } n \rightarrow \infty.
$$
\end{theorem}
For related ensembles of random covariance matrices it is well known that spectral distribution function of eigenvalues converges to the Marchenko--Pastur law.
In this case G{\"o}tze and Tikhomirov in~\cite{GotTikh2004} received similar results to~\cite{GotTikh2006}. Recently Adamczak,~\cite{Adamczak2011}, proved the Marchenko--Pastur law for matrices with martingale structure.
He assumed that the matrix elements have moments of all orders and imposed conditions similar to~\eqref{eq: variance prop 2}. Another class of random matrices with dependent entries
was considered in~\cite{ORourke2012} by O'Rourke. In a forthcoming paper we prove analogs of these theorems for random covariance matrices.

The paper organized as follows. In Section~2 we give a proof of Theorem~\ref{th:universality} using the method of Stieltjes transforms. In Section~3 we prove Theorem~\ref{th: gaussian case} by the classical moment method.

Throughout this paper we assume that all random variables are defined on common probability
space $(\Omega, \mathcal{F}, \Pb)$. Let $\Tr(\A)$ denote the trace of a matrix $\A$. For a vector $x = (x_1, ... , x_n)$ let $||x||_2 := (\sum_{i=1}^n x_i^2 )^{1/2}$.
We denote the operator norm of the matrix $\A$ by $||\A||:=\sup\limits_{||x||_2 = 1} ||\A x||_2$. We will write $a \le_m b$ if there is an absolute constant $C$ depending on $m$ only such that $a \le C b$.

From now on we shall omit the index $n$ in the notation for random matrices.

\section{Proof of Theorem~\ref{th:universality}}

We denote the Stieltjes transforms of $F^{\X}$ and $F^{\Y}$ by $S^{\X}(z)$ and $S^{\Y}(z)$ respectively.
Due to the relations between distribution functions and Stieltjes transforms, the statement of Theorem~\ref{th:universality}
will follow from
\begin{equation} \label{eq: st convergenes}
|S^{\X}(z) - S^{\Y}(z)| \rightarrow 0 \text{ as } n \rightarrow \infty.
\end{equation}
Set
\begin{equation} \label{eq:resolvent}
\RR^{\X}(z):= \left(\frac{1}{\sqrt{n}}\X - z \I \right)^{-1} \text{ and } \RR^{\Y}(z):= \left(\frac{1}{\sqrt{n}}\Y - z \I\right)^{-1}.
\end{equation}
By definition
\begin{equation*}
S^{\X}(z ) = \frac{1}{n} \Tr \E\RR^{\X}(z) \text{ and } S^{\Y}(z ) = \frac{1}{n} \Tr \E\RR^{\Y}(z).
\end{equation*}

We divide the proof of~\eqref{eq: st convergenes} into the two subsections~\ref{sec:transcation} and~\ref{sec:universality}.

Note that we can substitute $\tau$ in~\eqref{eq: limit of Lind. ratio}  by a decreasing sequence $\tau_n$ tending to zero such that
\begin{equation}\label{eq: LF2}
L_n(\tau_n) \rightarrow 0 \text{ as } n \rightarrow \infty.
\end{equation}
and $\lim_{n \rightarrow \infty} \tau_n \sqrt n = \infty$.

\subsection{Truncation of random variables}\label{sec:transcation}
In this section we truncate the elements of the matrices $\X$ and $\Y$.
Let us omit the indices $\X$ and $\Y$ in the notations of the resolvent and the Stieltjes transforms.

Consider some symmetric $n\times n$ matrix $\D$. Put $\tilde \X = \X + \D$. Let
$$
\tilde \RR = \left (\frac{1}{\sqrt n} \tilde \X - z \I \right )^{-1}.
$$
\begin{lemma} \label{l:tcn}
$$
|\Tr \RR - \Tr \tilde \RR| \le \frac{1}{v^2} (\Tr \D^2)^{\frac{1}{2}}.
$$
\end{lemma}
\begin{proof}
By the resolvent equation
\begin{equation}\label{eq: resolvent equation}
\RR = \tilde  \RR -\frac{1}{\sqrt n} \RR \D \tilde \RR.
\end{equation}
For resolvent matrices we have, for $z = u + i v, v>0$,
\begin{equation} \label{eq: resolvent inequaility}
\max\{||\RR||, ||\tilde  \RR||\} \le \frac{1}{v}.
\end{equation}
Using~\eqref{eq: resolvent equation} and~\eqref{eq: resolvent inequaility} it is easy to show that
\begin{equation*}\label{eq: pertubation}
|\Tr \RR - \Tr \tilde \RR| = \frac{1}{ \sqrt n} |\Tr \RR \D \tilde \RR| \le \frac{1}{v^2} (\Tr \D^2)^{\frac{1}{2}}.
\end{equation*}
\end{proof}
We split the matrix entries as $X = \hat X + \check X$, where $\hat X := X \one(|X| < \tau_n \sqrt n)$ and
$\check X := X \one(|X| \geq \tau_n \sqrt n)$. Define the matrix $\hat \X = \{\hat X_{ij}\}_{i,j=1}^n$. Let
\begin{equation*}
\hat \RR(z):= \left ( \frac{1}{\sqrt n} \hat \X - z \I \right )^{-1} \text{ and } \hat S(z) = \frac{1}{n} \E \Tr \hat \RR(z).
\end{equation*}
By Lemma~\ref{l:tcn}
$$
|S(z) - \hat S(z)| \le \frac{1}{v^2} \left( \frac{1}{n^2} \sum_{i,j=1}^n \E X_{ij}^2 \one(|X_{ij}| \geq \tau_n \sqrt n) \right )^{1/2} = v^{-2}L_n^{\frac{1}{2}}(\tau_n).
$$
From~\eqref{eq: LF2} we conclude that
$$
|S(z) - \hat S(z)| \rightarrow 0 \text{ as } n \rightarrow \infty.
$$
Introduce the centralized random variables $\overline X_{ij} = \hat X_{ij} - \E(\hat X_{ij} | \mathfrak{F}^{(i,j)})$ and the matrix $\overline \X = \{ \overline X_{ij}\}_{i,j = 1 }^n$. Let
\begin{equation*}
\overline \RR(z):= \left ( \frac{1}{\sqrt n} \overline \X - z \I \right )^{-1} \text{ and } \overline S(z) = \frac{1}{n} \E \Tr \overline \RR(z).
\end{equation*}
Again by Lemma~\ref{l:tcn}
$$
|\hat S(z) - \overline S(z)| \le \frac{1}{v^2} \left( \frac{1}{n^2} \sum_{i,j=1}^n \E X_{ij}^2 \one(|X_{ij}| \geq \tau_n \sqrt n) \right )^{1/2} = v^{-2}L_n^{\frac{1}{2}}(\tau_n).
$$
In view of~\eqref{eq: LF2} the right hand side tends to zero as $n \rightarrow \infty$.

Now we show that~\eqref{eq: variance prop} will hold if we replace $\X$ by $\overline \X$. For all $1 \le i \le j \le n$
\begin{align*}
&\E(\overline X_{ij}^2|\mathfrak{F}^{(i,j)}) - \sigma_{ij}^2  \\
&= \E(\overline X_{ij}^2|\mathfrak{F}^{(i,j)}) - \E(\hat X_{ij}^2|\mathfrak{F}^{(i,j)}) - \E(\check X_{ij}^2|\mathfrak{F}^{(i,j)}) + \E(X_{ij}^2|\mathfrak{F}^{(i,j)})- \sigma_{ij}^2 \nonumber.
\end{align*}
By the triangle inequality and~\eqref{eq: variance prop},~\eqref{eq: LF2}
\begin{align} \label{eq: variance prop for  trunctated}
&\frac{1}{n^2} \sum_{i,j = 1}^n \E|\E(\overline X_{ij}^2|\mathfrak{F}^{(i,j)}) - \sigma_{ij}^2|  \\\
&\le \frac{1}{n^2} \sum_{i,j = 1}^n \E|\E(X_{ij}^2|\mathfrak{F}^{(i,j)}) - \sigma_{ij}^2| + 2 L_n(\tau_n) \rightarrow 0 \text{ as } n \rightarrow \infty. \nonumber
\end{align}

It is also not very difficult to check that the condition~\eqref{eq: variance prop 2} holds true for the matrix $\X$ replaced by $\overline \X$.

Similarly, one may truncate the elements of the matrix $\Y$ and consider the matrix $\overline \Y$ with the entries $Y_{ij} \one(|Y_{ij}| \le \tau_n \sqrt n)$. Then one may check that
\begin{equation} \label{eq: variance prop for  trunctated gauss}
\frac{1}{n^2} \sum_{i, j=1}^n|\E \overline Y_{ij}^2 - \sigma_{ij}^2| \rightarrow 0 \text{ as } n \rightarrow \infty.
\end{equation}

In what follows assume from now on that $|X_{ij}| \le \tau_n \sqrt n$ and $|Y_{ij}| \le \tau_n \sqrt n$. We shall write $\X, \Y$ instead of $\overline \X$ and $\overline \Y$ respectively.

\subsection{Universality of the spectrum of eigenvalues}\label{sec:universality}

To prove~\eqref{eq: st convergenes} we will use a method introduced in~\cite{Bentkus2003}. Define the matrix $\Z: = \Z(\varphi) :=  \X \cos \varphi +  \Y \sin \varphi$.
It is easy to see that $\Z(0) = \X$ and $\Z(\pi/2) = \Y$.
Set $\W := \W(\varphi): = n^{-1/2} \Z$ and
\begin{equation}
\RR(z, \varphi):= (\W - z \I)^{-1}.
\end{equation}
Introduce the Stieltjes transform
$$
S(z, \varphi) := \frac{1}{n} \sum_{i=1}^{n} \E [\RR(z, \varphi)]_{i i}.
$$
Note that $S(z, 0)$ and $S(z, \pi/2)$ are the Stieltjes transforms $S^{\X}(z)$ and
$S^{\Y}(z)$ respectively.

Obviously we have
\begin{equation} \label{eq:NL}
S(z,\frac{\pi}{2}) - S(z,0) = \int_0^\frac{\pi}{2} \frac{\partial S(z, \varphi)}{\partial \varphi}d\varphi.
\end{equation}
To simplify the arguments we will omit arguments in the notations of matrices and Stieltjes transforms. We have
\begin{align*}
\frac{\partial \W}{\partial \varphi} =
\frac{1}{\sqrt n} \sum_{i = 1}^n \sum_{j = 1}^n \frac{\partial Z_{ij}}{\partial \varphi} \ee_i \ee_{j}^T,
\end{align*}
where we denote by $\ee_i$ the column vector with $1$ in position $i$ and zeros in the other positions.
We may rewrite the integrand in~\eqref{eq:NL} in the following way
\begin{align} \label{eq: derivative}
&\frac{\partial S}{\partial \varphi} = -\frac{1}{n} \E \Tr \RR \frac{\partial \W}{\partial \varphi} \RR \\
&= -\frac{1}{n^{3/2}} \sum_{i = 1}^n \sum_{j=1}^n \E\Tr \RR \frac{\partial Z_{ij}}{\partial \varphi} \ee_i \ee_{j}^T \RR  \nonumber \\
&= \frac{1}{n^{3/2}} \sum_{i = 1}^n \sum_{j=1}^n \E \frac{\partial Z_{ij}}{\partial \varphi} u_{i j}, \nonumber
\end{align}
where $u_{ij} = -[\RR^2]_{ji}$.

For all $1 \le i \le j \le n$ introduce the random variables
\begin{equation*}
\xi_{ij} :=  Z_{ij}, \quad \hat \xi_{ij} := \frac{\partial Z_{ij}}{\partial \varphi} = - \sin \varphi X_{i j} + \cos \varphi Y_{i j},
\end{equation*}
and the sets of random variables
$$
\xi^{(ij)} := \{\xi_{kl}: 1 \le k \le l \le n,  (k,l) \neq (i,j) \}.
$$
Using Taylor's formula one may write
\begin{align*} \label{eq: TF}
u_{i j} (\xi_{i j}, \xi^{(ij)}) = u_{i j}(0, \xi^{(ij)}) + \xi_{i j} \frac{\partial u_{i j}}{\partial \xi_{i j}}(0,\xi^{(ij)}) +
\E_{\theta} \theta (1-\theta) \xi_{i j}^2 \frac{\partial^2 u_{i j}}{\partial \xi_{i j}^2}(\theta \xi_{i j}, \xi^{(ij)}),
\end{align*}
where $\theta$ has a uniform distribution on $[0,1]$ and is independent of $(\xi_{i j},\xi^{(ij)})$. Multiplying both sides of the last equation by $\hat \xi_{i j}$
and taking mathematical expectation on both sides we have
\begin{align}
&\E \hat \xi_{i j} u_{i j} (\xi_{i j}, \xi^{(ij)}) =  \E \hat \xi_{i j} u_{i j}(0, \xi^{(ij)}) + \E\hat \xi_{i j}\xi_{i j}\frac{\partial u_{i j}}{\partial \xi_{i j}}(0,\xi^{(ij)}) \\
&+ \E \theta (1-\theta)\hat \xi_{i j} \xi_{i j}^2  \frac{\partial^2 u_{i j}}{\partial \xi_{i j}^2}(\theta \xi_{i j}, \xi^{(ij)}). \nonumber
\end{align}
By independence of $Y_{ij}$ and $\xi^{(ij)}$ we get
\begin{equation} \label{eq: first term_1}
\E Y_{ij} u_{i j}(0, \xi^{(ij)}) = \E Y_{ij} \E u_{i j}(0, \xi^{(ij)}) = 0.
\end{equation}
By the properties of conditional expectation and condition~\eqref{eq:mrtingale prop}
\begin{equation} \label{eq: first term_2}
\E X_{ij} u_{i j}(0, \xi^{(ij)}) = \E u_{i j}(0, \xi^{(ij)}) \E ( X_{i j}  | \mathfrak{F}^{(i,j)}) = 0.
\end{equation}
By~\eqref{eq: TF},~\eqref{eq: first term_1} and~\eqref{eq: first term_2} we can rewrite~\eqref{eq: derivative} in the following way
\begin{align*}
&\frac{\partial S}{\partial \varphi} = \frac{1}{n^2} \sum_{i,j = 1}^n \E \hat \xi_{i j}\xi_{i j}\frac{\partial u_{i j}}{\partial \xi_{i j}}(0,\xi^{(ij)}) +
\frac{1}{n^2} \sum_{i,j = 1}^n  \E \theta (1-\theta)\hat \xi_{i j} \xi_{i j}^2  \frac{\partial^2 u_{i j}}{\partial \xi_{i j}^2}(\theta \xi_{i j}, \xi^{(ij)}) \\
& = \mathbb A_1 + \mathbb A_2.
\end{align*}
It is easy to see that
$$
\hat \xi_{ij} \xi_{i j} = - \frac{1}{2} \sin 2 \varphi X_{ij}^2 + \cos^2 \varphi X_{i j} Y_{i j} - \sin^2 \varphi X_{i j} Y_{i j} + \frac{1}{2} \sin 2 \varphi Y_{ij}^2.
$$
The random variables $Y_{ij}$ are independent of $X_{ij}$ and $\xi^{(ij)}$. Using this fact we conclude that
\begin{align} \label{eq: second term 1}
&\E X_{i j} Y_{i j} \frac{\partial u_{i j}}{\partial \xi_{i j}}(0,\xi^{(ij)}) = \E Y_{ij} \E X_{i j} \frac{\partial u_{i j}}{\partial \xi_{i j}}(0,\xi^{(ij)})  = 0, \\
&\E Y_{ij}^2 \frac{\partial u_{ij}}{\partial \xi_{ij}}(0, \xi^{(ij)}) = \sigma_{ij}^2 \E \frac{\partial u_{ij}}{\partial \xi_{ij}}(0, \xi^{(ij)}).
\end{align}
By the properties of conditional mathematical expectation we get
\begin{equation}\label{eq: second term 2}
\E X_{ij}^2\frac{\partial u_{ij}}{\partial \xi_{ij}}(0, \xi^{(ij)})
= \E \frac{\partial u_{ij}}{\partial \xi_{ij}}(0, \xi^{(ij)})\E (X_{ij}^2|\mathfrak{F}^{(i,j)}).
\end{equation}
A direct calculation shows that the derivative of $u_{ij} = -[\RR^2]_{ji}$ is equal to
\begin{align*}
&\frac{\partial u_{ij}}{\partial \xi_{ij}} =  \left [\RR^2 \frac{\partial \Z }{\partial \xi_{ij}} \RR \right ]_{ji}  +  \left [ \RR \frac{\partial \Z }{\partial \xi_{ij}} \RR^2 \right ]_{ji}  \\
&=\frac{1}{\sqrt n} [ \RR^2 \ee_i \ee_{j}^T \RR]_{ji} + \frac{1}{\sqrt n} [\RR^2 \ee_j \ee_{i}^T \RR]_{ji}
+\frac{1}{\sqrt n}[\RR \ee_i \ee_{j}^T \RR^2]_{ji} + \frac{1}{\sqrt n} [ \RR \ee_j \ee_{i}^T \RR^2]_{ji} \\
&= \frac{1}{\sqrt n} [\RR^2]_{ji}[\RR]_{ji} + \frac{1}{\sqrt n} [\RR^2]_{jj}[\RR]_{ii}
+ \frac{1}{\sqrt n} [\RR]_{ji}[\RR^2]_{ji} + \frac{1}{\sqrt n} [\RR]_{jj}[\RR^2]_{ii}.
\end{align*}
Using the obvious bound for the spectral norm of the matrix resolvent $||\RR|| \le v^{-1}$ we get
\begin{equation} \label{eq: bound of first derivative}
\left | \frac{\partial u_{ij}}{\partial \xi_{ij}} \right | \le \frac{C}{\sqrt n v^3}.
\end{equation}
From~\eqref{eq: second term 1}--\eqref{eq: bound of first derivative} and~\eqref{eq: variance prop for  trunctated}--\eqref{eq: variance prop for  trunctated gauss} we deduce
\begin{equation}\label{eq: A_1 est}
|\mathbb A_1| \le \frac{C}{n^2 v^3} \sum_{i,j = 1}^n \E|\E(X_{ij}^2|\mathfrak{F}^{(i,j)})-\sigma_{ij}^2| \rightarrow 0 \text{ as } n \rightarrow \infty.
\end{equation}
It remains to estimate $\mathbb A_2$. We calculate the second derivative of $u_{ij}$
\begin{align*}
&\frac{\partial^2 u_{ij}}{\partial \xi_{ij}^2} = - 2\left [\RR^2 \frac{\partial \W }{\partial \xi_{ij}} \RR \frac{\partial \W}{\partial \xi_{ij}} \RR \right ]_{ji}
- 2\left [ \RR \frac{\partial \W }{\partial \xi_{ij}} \RR^2 \frac{\partial \W}{\partial \xi_{ij}} \RR \right ]_{ji} \\
&- 2\left [ \RR \frac{\partial \W }{\partial \xi_{ij}} \RR \frac{\partial \W}{\partial \xi_{ij}} \RR^2 \right ]_{ji}
 = \mathbb T_{1} + \mathbb T_{2} + \mathbb T_{3}.
\end{align*}
Let's expand the term $\mathbb T_{1}$
\begin{align} \label{eq: A111}
\mathbb T_{1} = - 2\left [ \RR^2 \frac{\partial \W }{\partial \xi_{ij}} \RR \frac{\partial \W}{\partial \xi_{ij}} \RR \right ]_{ji}=
\mathbb T_{11} + \mathbb T_{12} + \mathbb T_{13} + \mathbb T_{14},
\end{align}
where we denote
\begin{align*}
&\mathbb T_{11} = -\frac{2}{n} [\RR^2]_{ji} [\RR]_{ji} [\RR]_{ji}, \quad \mathbb T_{12} = -\frac{2}{n} [\RR^2]_{ji} [\RR]_{jj} [\RR]_{ii}, \\
&\mathbb T_{13} = -\frac{2}{n} [\RR^2]_{jj} [\RR]_{ii} [\RR]_{ji}, \quad \mathbb T_{14} = -\frac{2}{n} [\RR^2]_{jj} [\RR]_{ij} [\RR]_{ii}.
\end{align*}
Using again the bound $||\RR|| \le v^{-1}$ we can show that
$$
\max (|\mathbb T_{11}|,|\mathbb T_{12}|,|\mathbb T_{13}|,|\mathbb T_{14}|) \le \frac{C}{n v^4}.
$$
From the expansion~\eqref{eq: A111} and the bounds of $\mathbb T_{1i}, i = 1,2,3,4$ we conclude that
$$
|\mathbb T_{1}| \le \frac{C}{n v^4}.
$$
Repeating the above arguments one can show that
$$
\max (|\mathbb T_{2}|,|\mathbb T_{3}|) \le \frac{C}{n v^4}.
$$
Finally we have
$$
\left|\frac{\partial^2 u_{ij}}{\partial \xi_{ij}^2}(\theta \xi_{i j},\xi^{(i j)}) \right| \le \frac{C}{n v^4}.
$$
Using the assumption $|\xi_{ij}| \le \tau_n \sqrt n$ and the condition~\eqref{eq: variance prop 2} we deduce the bound
\begin{equation}\label{eq: A_2 est}
|\mathbb A_2| \le \frac{C \tau_n}{v^4}.
\end{equation}
We may turn $\tau_n$ to zero and conclude the statement of Theorem~\ref{th:universality} from~\eqref{eq:NL}, \eqref{eq: derivative}, \eqref{eq: A_1 est} and~\eqref{eq: A_2 est}.

\section{Proof of Theorem~\ref{th: gaussian case}}

We prove the theorem using the moment method. It is easy to see that the moments of $F^{\Y}(x)$ can be rewritten as normalized traces of powers of $\Y$:
$$
\int_\R x^k dF^{\Y}(x) = \E \frac{1}{n} \Tr \left ( \frac{1}{\sqrt n} \Y \right)^k, ~  k \geq 1.
$$
It is sufficient to prove that
$$
\E \frac{1}{n} \Tr \left ( \frac{1}{\sqrt n} \Y \right)^k = \int_\R x^k dG(x) + o_k(1).
$$
for $k \geq 1$, where $o_k(1)$ tends to zero as $n \rightarrow \infty$ for any fixed $k$.

It is well known that the moments of semicircle law are given by the Catalan numbers
$$
\beta_k = \int_\R x^k dG(x) = \begin{cases}
  \frac{1}{m+1} \binom{2m}m, & k = 2m \\
  0, & k = 2m + 1.
\end{cases}
$$

Furthermore we shall use the notations and the definitions from~\cite{BaiSilv2010}.
A graph is a triple $(E, V, F)$, where $E$ is the set of edges, $V$ is the set of vertices, and $F$ is a function, $F: E \rightarrow V \times V$.
Let $\vect i = (i_1, ... , i_k)$ be a vector taking values in $\{1, ..., n\}^k$. For a vector $\vect i$ we define a $\Gamma$-graph as follows. Draw a horizontal line and plot the numbers $i_1, ... , i_k$ on it. Consider the distinct numbers as vertices, and draw $k$ edges $e_j$ from $i_j$ to $i_{j+1}, j = 1,...,k$, using $i_{k+1} = i_1$ by convention. Denote the number of distinct $i_j$'s by $t$. Such a graph is called a $\Gamma(k,t)$-graph.


Two $\Gamma(k,t)$-graphs are said to be isomorphic if they can be converted each other by a permutation of $(1, ..., n)$. By this definition, all $\Gamma$-graphs are classified into isomorphism classes. We shall call the $\Gamma(k,t)$-graph canonical if it has the following properties:\\
1) Its vertex set is $\{1, .... , t\}$;\\
2) Its edge set is $\{e_1, ... , e_k \}$;\\
3) There is a function g from $\{1, ..., k\}$ onto $\{1, ... , t\}$ satisfying $g(1) = 1$ and $g(i) \le \max\{g(1), ... , g(i-1)\} + 1$ for $1 < i \le k$; \\
4) $F(e_i) = (g(i), g(i+1))$, for $i=1,...,k$, with the convention $g(k+1) = g(1) = 1$.

It is easy to see that each isomorphism class contains one and only one canonical $\Gamma$-graph that is associated with a function $g$, and a general graph in this class can be defined by $F(e_j) = (i_{g(j)}, i_{g(j+1)})$. It is easy to see that each isomorphism class contains $n(n-1)...(n-t+1)$ $\Gamma(k,t)$-graphs.

We shall classify all canonical graphs into three categories. Category~$1$ consists of all canonical $\Gamma(k,t)$-graphs with the property that each
edge is coincident with exactly one other edge of opposite direction and the graph of noncoincident edges forms a tree. It is easy to see if $k$ is odd then there are no graphs in category~1. If $k$ is even, i.e. $k = 2m$, say, we denote a
$\Gamma(k,t)$-graph by $\Gamma_1(2m)$. Category~$2$ consists of all canonical graphs that have at least one edge with odd multiplicity. We shall denote the graph from this category by $\Gamma_2(k,t)$. Finally, category~$3$
consists of all other canonical graphs, which we denote by $\Gamma_3(k,t)$.

It is known, see~\cite[Lemma~2.4]{BaiSilv2010}, that the number of $\Gamma_1(2m)$-graphs is equal to $\frac{1}{m+1} \binom{2m}m$.

We expand the traces of powers of $\Y$ in a sum
\begin{equation} \label{eq: moments of X}
\Tr \left ( \frac{1}{\sqrt n} \Y \right)^k = \frac{1}{n^{k/2}}\sum_{i_1, i_2, ... , i_k} Y_{i_1 i_2} Y_{i_2 i_3} ... Y_{i_{k} i_1},
\end{equation}
where the summation is taken over all sequences $\vect i = (i_1, ... , i_k) \in \{1, ... , n \}^k$.

For each vector $\vect i$ we construct a graph $G(\vect i)$ as above. We denote by  $Y(\vect i) = Y(G(\vect i))$.

Then we may split the moments of $F^{\Y}(x)$ into three terms
$$
\E \frac{1}{n} \Tr \left ( \frac{1}{\sqrt n} \Y \right)^k = \frac{1}{n^{k/2+1}}\sum_{\vect i} \E Y_{i_1 i_2} Y_{i_2 i_3} ... Y_{i_{k} i_1} = S_1 + S_2 + S_3,
$$
where
$$
S_j = \frac{1}{n^{k/2+1}} \sum_{\Gamma(k,t) \in C_j} \sum_{G(\vect i) \in \Gamma(k,t)} \E[Y(G(\vect i))],
$$
and the summation $ \sum_{\Gamma(k,t) \in C_j}$ is taken over all canonical $\Gamma(k,t)$-graphs in category $C_j$ and the summation $\sum_{G(\vect i) \in \Gamma(k,t)}$ is taken over all isomorphic graphs for a given canonical graph.

From the independence of $Y_{ij}$ and $\E Y_{ij}^{2s-1} = 0, s \geq 1,$ it follows that $S_2 = 0$.

For the graphs from categories $C_1$ and $C_3$ we introduce further notations. Let us consider the $\Gamma(k,t)$-graph $G(\vect i)$.
Without loss of generality we assume that $i_l, l = 1, ..., t$ are distinct coordinates of the
vector $\vect i$ and define a vector $\vect i_t = (i_1, ... , i_t)$. We also set $G(\vect i_t):=G(\vect i)$.
Let $\tilde{\vect \imath}_t = (i_1, ... , i_{q-1}, i_{q+1}, ... , i_{t})$ and
$\hat{\vect \imath}_t = (i_1, ... , i_{p-1},i_{p+1}, ... , i_{q-1},i_{q+1},...,i_{t})$ be vectors derived from $\vect i_t$ by deleting the elements in the position $q$ and
$p, q$ respectively. We additionally assume that the coordinates of $\hat{\vect \imath}_t$ do not coincide with $i_p$. We denote the graph without the vertex $i_q$ and all edges linked to it by $G(\tilde{\vect \imath}_t)$. If the vertex $i_q$ is incident to a loop we denote by $G'(\vect i_t)$ the graph with this loop removed. By $\tilde{G}(\vect i_t)$ we mean the graph derived from $G(\vect i_t)$ by deleting the edge from $i_p$ to $i_q$ taking into account the multiplicity.

Now we will estimate the term $S_3$. For a graph from category $C_3$ we know that $k$ has to be even, i.e. $k = 2m$, say. We illustrate the example of a $\Gamma_3(k,t)$-graph in Figure~\ref{fig:graph3}. This graph corresponds to the term $Y(G(\vect i_3)) = Y_{i_1 i_1}^2 Y_{i_1 i_2}^2 Y_{i_2 i_3}^4 Y_{i_3 i_3}^2$.
\begin{figure}
\begin{center}
\scalebox{.35}{\includegraphics{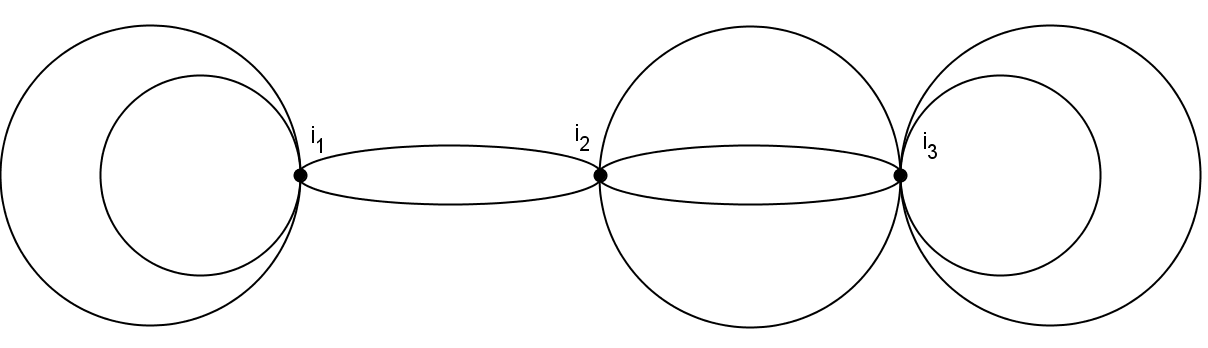}}
\end{center}
\caption{Graph $\Gamma_3(10,3)$.}
\label{fig:graph3}
\end{figure}

We mention that $\E Y_{i_p i_q}^{2s} \le_s \sigma_{i_p i_q}^{2s}$. Hence we may rewrite the terms which correspond to the graphs from category $C_3$ via variances.

Each graph $G$ from $C_3$ we can decompose into two graphs $G = G_1 \cup G_2$ in the following way. We will paint all edges into two colours such that all coincident edges have the same colour. We choose the graph $G_1$ such that one of the following cases holds true:\\
i) graph $G_1$ consists from only one vertex and all incident to it loops;\\
ii) graph $G_1$ consists from two vertices and the edge between them with the multiplicity greater then two; \\
iii) Each edge of the graph $G_1$ coincide with exactly one edge of the opposite direction and the graph of non coincident edges forms a simple cycle.

We may assume that the sum of multiplicities of all edges from $G_1$ is equal $2s$. It remains to consider the remaining $2(m - s)$ edges from the graph $G_2$.

Denote by $V_1, E_1$ and $V_2, E_2$ the sets of coordinates and edges of the graphs $G_1$ and $G_2$ respectively. We define $\vect i_t^1: = (i_l, i_l \in V_1)$. Now we may fix the graph $G_1$ and consider the graph $G_2$.

a) If the set $E_2$ is empty, then one should go to the step c). Otherwise we can consider the following opportunities:
\begin{enumerate}
\item {There is a loop from $G_2$ incident to the vertex $i_q \in V_2$ with the multiplicity $2a, a \geq 1$. In this case we estimate
\begin{equation*}
\qquad \frac{1}{n^{m+1}}\sum_{G(\vect i) \in \Gamma(2m,t)} \E[Y(G(\vect i))] \le_a \frac{1}{n^{m+1}}\sum_{\vect i_t} \E[Y(G'(\vect i_t))] \sigma_{i_q i_q}^{2a}.
\end{equation*}
Applying $a$ times the inequality $n^{-1} \sigma_{i_q i_q}^2 \le B_{i_q}^2$ and the condition~\eqref{eq: variance prop 3}, we delete all loops incident to this vertex;}
\item {There are no loops incident to the vertex $i_q \in V_2 \setminus V_1$, but $i_q$ is connected with only one vertex $i_p \in V_2$ by an edge of the graph $G_2$ and the multiplicity of this edge is equal $2b, b \geq 1$. In this case we estimate
\begin{equation*}
\qquad \qquad\frac{1}{n^{m+1}}\sum_{G(\vect i) \in \Gamma(2m,t)} \E[Y(G(\vect i))] \le_b \frac{1}{n^{m+1}}\sum_{\tilde{\vect \imath}_t} \E[Y(G(\tilde{\vect \imath}_t))] \sum_{i_q=1}^n\sigma_{i_p i_q}^{2b}.
\end{equation*}
Here we may use $b-1$ times the inequality $n^{-1} \sigma_{i_p i_q}^2 \le B_{i_p}^2$ and condition~\eqref{eq: variance prop 3} and consequently delete all coinciding edges except two. Then we may again apply condition~\eqref{eq: variance prop 3} to delete $i_q$;}
\item {There are no loops incident to some vertex from $V_2$ and no vertices in $V_2\setminus V_1$ which are connected with only one vertex from $V_2$. Then we can take any two vertices, let's say $i_p$ and $i_q$ from $V_2$ and estimate
\begin{equation*}
\qquad \qquad\frac{1}{n^{m+1}}\sum_{G(\vect i) \in \Gamma(2m,t)} \E[Y(G(\vect i))] \le_b \frac{1}{n^{m+1}}\sum_{\vect i_t} \E[Y(\tilde{G}(\vect i_t))] \sigma_{i_p i_q}^{2c},
\end{equation*}
where the multiplicity of the edge between $i_p$ and $i_q$ is equal $2c, c \geq 1$. Here we may use $c$ times the inequality $n^{-1} \sigma_{i_p i_q}^2 \le B_{i_p}^2$ and the condition~\eqref{eq: variance prop 3} and consequently delete all coinciding edges between $i_p$ and $i_q$;}
\end{enumerate}
b) go to step a);\\
c) It is easy to see that each time on the step a) we use the same bound~\eqref{eq: variance prop 3}. Hence we will have
\begin{equation}\label{eq: estimate for the term C 3}
\frac{1}{n^{m+1}}\sum_{G(\vect i) \in \Gamma(2m,t)} \E[Y(G(\vect i))] \le_m \frac{C^{2(m-s)} n^{m-s}}{n^{m+1}} \sum_{\vect i_t^1}\E[Y(G_1(\vect i_t^1))].
\end{equation}
It remains to estimate the right hand side of~\eqref{eq: estimate for the term C 3}. In the case i) we may estimate
\begin{align*}
&\frac{1}{n^{s+1}} \sum_{i_1=1}^n \E Y_{i_1 i_1}^{2s} \le_s \frac{1}{n^{s+1}} \sum_{i_1=1}^n \sigma_{i_1 i_1}^{2s} \le \frac{C^{2(s-1)}}{n^{2}} \sum_{i_1=1}^n \sigma_{i_1 i_1}^2  \\
&\le C^{2(s-1)}  \tau_n^2  + \frac{C^{2(s-1)} }{n^2} \sum_{i_1=1}^n \E Y_{i_1 i_1}^2 \one(|Y_{i_1 i_1}| \geq \tau_n \sqrt n )  \nonumber\\
&\le C^{2(s-1)} \tau_n^2 + C^{2(s-1)}  L_n(\tau_n) = o_s(1), \nonumber
\end{align*}
where we have used the inequality $n^{-1}\sigma_{i_1 i_1}^2 \le B_{i_1}^2$  and~\eqref{eq: variance prop 3}.

In the case ii) we will use the bound
\begin{align*}
&\frac{1}{n^{s+1}}\sum_{\substack{i_1,i_2=1\\i_1 \neq i_2}}^n \E Y_{i_1 i_2}^{2s} \le_s \frac{1}{n^{s+1}}\sum_{i_1,i_2=1}^n \sigma_{i_1 i_2}^{2s} \le \frac{C^{2(s-2)}}{n^{3}} \sum_{i_1,i_2=1}^n \sigma_{i_1 i_2}^4  \\
&\le C^{2(s-2)} \frac{\tau_n^2}{n^2} \sum_{i_1,i_2=1}^n \sigma_{i_1 i_2}^2 +\frac{C^{2(s-2)}}{n^3} \sum_{i_1,i_2=1}^n \sigma_{i_1 i_2}^2 \E Y_{i_1 i_2}^2 \one(|Y_{i_1 i_2}| \geq \tau_n \sqrt n) \nonumber \\
&\le C^{2(s-1)} \tau_n^2 + C^{2(s-1)} L_n(\tau_n) = o_s(1), \nonumber
\end{align*}
where we have used the inequality $n^{-1} \sigma_{i_1 i_2}^2 \le B_{i_1}^2$ and~\eqref{eq: variance prop 3}.

It remains to consider the case iii) only. We need to introduce further
notation for this situation. We may redenote the vertices from the set $V_1$
and assume that $\vect i_t^1 = \vect i_s:=(i_1, ... , i_s)$. By $G_1(\vect
i_s)$ we denote the graph $G_1(\vect i_t^1)$. Using the previous notations of $\tilde{\vect \imath}_s, \hat{\vect \imath}_s$ and $\tilde{G}_1(\vect i_s)$ we set $p = 1$ and $q = 2$. Finally by $\hat{G}_1(\tilde{\vect \imath}_s)$ we denote the graph derived from $\tilde{G}_1(\vect i_s)$ by deleting the vertex $i_2$ and the edge between $i_2$ and some another vertex $i_x, 2 < x \le s$. We may write
\begin{align}\label{eq: estimate for the term C 3 part 1}
&\frac{1}{n^{s+1}}\sum_{\vect i_s} \E Y(G_1(\vect i_s)) \le \frac{\tau_n^2}{n^{s}} \sum_{\vect i_s}\E Y(\tilde{G}_1(\vect i_s)) \\
&+ \frac{1}{n^{s+1}} \sum_{\vect i_s}\E Y(\tilde{G}_1(\vect i_s)) \E Y_{i_1 i_2}^2 \one(|Y_{i_1 i_2}| \geq \tau_n \sqrt n).\label{eq: estimate for the term C 3 part 2}
\end{align}
First we estimate the right hand side of~\eqref{eq: estimate for the term C 3 part 1}. The graph of the non coincident edges of $\tilde{G}_1(\vect i_s)$ forms a tree. We can sequently delete all vertices and edges from $\tilde{G}_1(\vect i_s)$ using the assumption~\eqref{eq: variance prop 3} on each step. We derive the bound
\begin{equation}\label{eq: estimate for the term C 3 part 3}
\frac{\tau_n^2}{n^{s}} \sum_{\vect i_s}\E Y(\tilde{G}_1(\vect i_s) \le C^{2(s-1)} \tau_n^2.
\end{equation}
For the term~\eqref{eq: estimate for the term C 3 part 2} we may write
\begin{align}\label{eq: estimate for the term C 3 part 4}
&\frac{1}{n^{s+1}} \sum_{\vect i_s}\E Y(\tilde{G}_1(\vect i_s)) \E Y_{i_1 i_2}^2 \one(|Y_{i_1 i_2}| \geq \tau_n \sqrt n) \\
&\le \frac{1}{n^{s+1}} \sum_{\vect i_s}\E Y(\hat{G}_1(\tilde{\vect \imath}_s)) \sigma_{i_2 i_x}^2 \E Y_{i_1 i_2}^2 \one(|Y_{i_1 i_2}| \geq \tau_n \sqrt n) \nonumber\\
&\le \frac{C^2}{n^{s}} \sum_{i_1, i_2 = 1}^n\E Y_{i_1 i_2}^2 \one(|Y_{i_1 i_2}| \geq \tau_n \sqrt n) \sum_{\hat{\vect \imath}_s} \E Y(\hat{G}_1(\tilde{\vect \imath}_s)) \nonumber.
\end{align}
Again using~\eqref{eq: variance prop 3} one may show that
\begin{equation}\label{eq: estimate for the term C 3 part 5}
\sum_{\hat{\vect \imath}_s} \E Y(\hat{G}_1(\tilde{\vect \imath}_s)) \le C^{2(s-2)}n^{s-2}.
\end{equation}
By~\eqref{eq: estimate for the term C 3 part 4} and~\eqref{eq: estimate for the term C 3 part 5} we have
\begin{equation}\label{eq: estimate for the term C 3 part 6}
\frac{1}{n^{s+1}} \sum_{\vect i_s}\E Y(\tilde{G}_1(\vect i_s)) \E Y_{i_1 i_2}^2 \one(|Y_{i_1 i_2}| \geq \tau_n \sqrt n) \le C^{2(s-1)} L_n(\tau_n).
\end{equation}
From~\eqref{eq: estimate for the term C 3 part 3} and~\eqref{eq: estimate for the term C 3 part 6} we derive the estimate
$$
\frac{1}{n^{s+1}}\sum_{\vect i_s} \E Y(G_1(\vect i_s)) \le C^{2(s-1)} \tau_n^2 + C^{2(s-1)} L_n(\tau_n).
$$
Finally for the cases i)--iii) we will have
$$
\frac{1}{n^{m+1}}\sum_{G(\vect i) \in \Gamma(2m,t)} \E[Y(G(\vect i))] \le_m C^{2(m-1)} (\tau_n^2 + L_n(\tau_n)) = o_m(1).
$$
As an example we recommend to check this algorithm for the graph in Figure~\ref{fig:graph3}.

It is easy to see that the number of different canonical graphs in $C_3$ is of order $O_m(1)$. Finally for the term $S_3$ we get
$$
S_3 =  o_m(1).
$$
It remains to consider the term $S_1$. For a graph from category $C_1$ we know that $k$ has to be even, i.e. $k = 2m$, say.
In the category $C_1$ using the notations of $\vect i_t, \tilde{\vect \imath}_t$ and $\hat{\vect \imath}_t$ we take $t = m+1$.

We illustrate on the left part of Figure~\ref{fig:graph} an example of the tree of noncoincident edges of a $\Gamma_1(2m)$-graph for $m = 5$. The term corresponding to this tree is
$Y(G(\vect i_6)) = Y_{i_1 i_2}^2 Y_{i_2 i_3}^2 Y_{i_2 i_4}^2 Y_{i_1 i_5}^2 Y_{i_5 i_6}^2$.
\begin{figure}
\begin{center}
\scalebox{.4}{\includegraphics{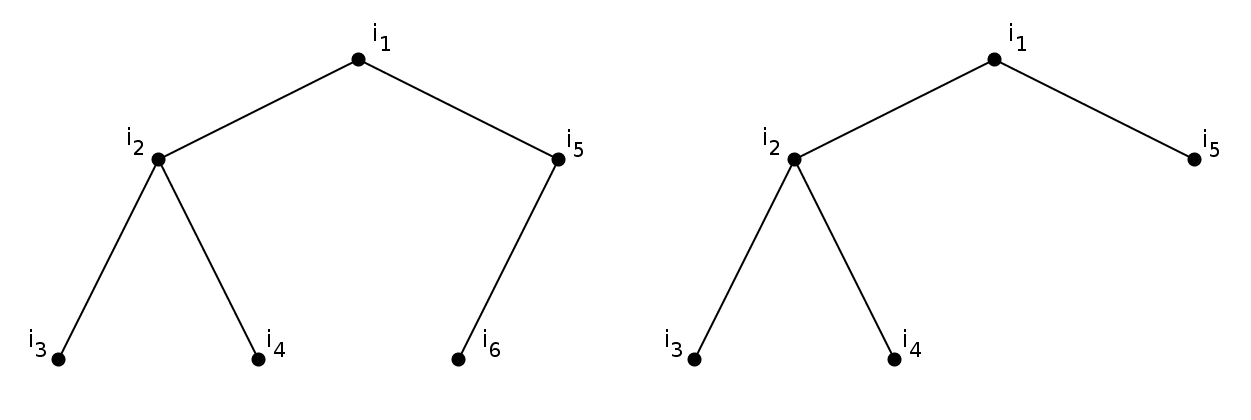}}
\end{center}
\caption{ On the left, the tree of noncoincident edges of a $\Gamma_1(10)$-graph is shown.
On the right, the tree of noncoincident edges of a $\Gamma_1(10)$-graph with deleted leaf $i_6$ is shown.}
\label{fig:graph}
\end{figure}

We denote by $\sigma^2(\vect i_{m+1}) = \sigma^2(G(\vect i_{m+1}))$ the product of $m$ numbers $\sigma_{i_s i_t}^2$,
where $i_s,i_t, s < t$ are vertices of the graph $G(\vect i_{m+1})$ connected by edges of this graph.
In our example, $\sigma^2(\vect i_{m+1}) = \sigma^2(\vect i_6) = \sigma_{i_1 i_2}^2 \sigma_{i_2 i_3}^2 \sigma_{i_2 i_4}^2 \sigma_{i_1 i_5}^2 \sigma_{i_5 i_6}^2$.

If $m = 1$ then $\sigma^2(\vect i_2) = \sigma_{i_1 i_2}^2$ and
\begin{equation}\label{eq: variance first step}
\frac{1}{n^{2}} \sum_{\substack{i_1, i_2 = 1 \\ i_1 \neq i_2}}^n \sigma_{i_1 i_2}^2 = \frac{1}{n}  \sum_{i_1 = 1}^n \left [\frac{1}{n} \sum_{i_2=1}^n \sigma_{i_1 i_2}^2 - 1 \right ] + 1 + o(1),
\end{equation}
where we have used $n^{-2} \sum_{i_1 = 1}^n \sigma_{i_1 i_1}^2 = o(1)$.
By~\eqref{eq: variance prop 2} the first term is of order $o(1)$. The number of canonical graphs in $C_1$ for $m = 1$ is equal to $1$.
We conclude for $m = 1$ that
$$
S_1 = n^{-2} \sum_{\Gamma_1(2)} \sum_{\substack{i_1, i_2 = 1 \\ i_1 \neq i_2}}^n \sigma_{i_1 i_2}^2  = 1 + o(1),
$$
Now we assume that $m > 1$. Let's consider the tree of non-coincident edges of the graph. We can find a leaf in the tree, let's say $i_q$, and a vertex $i_p$, which is connected to $i_q$ by an edge of this tree.
We have $\sigma^2(\vect i_{m+1}) = \sigma^2(\tilde{\vect \imath}_{m+1}) \cdot \sigma_{i_p i_q}^2$, where $\sigma^2(\tilde{\vect \imath}_{m+1}) = \sigma^2(G(\tilde{\vect \imath}_{m+1}))$.

In our example we can take the leaf $i_6$. On the right part of Figure~\ref{fig:graph}
we have drawn the tree with deleted leaf $i_6$. We have $\sigma_{i_p i_q}^2 = \sigma_{i_5 i_6}^2$ and $\sigma^2(\tilde{\vect \imath}_6) = \sigma_{i_1 i_2}^2 \sigma_{i_2 i_3}^2 \sigma_{i_2 i_4}^2 \sigma_{i_1 i_5}^2$.
\begin{align}
&\frac{1}{n^{m+1}} \sum_{\vect i_{m+1}} \sigma^2(\vect i_{m+1}) =\frac{1}{n^{m+1}}\sum_{\tilde{\vect \imath}_{m+1}} \sigma^2(\tilde{\vect \imath}_{m+1})\sum_{i_q = 1}^n \sigma_{i_p i_q}^2 + o_m(1) \nonumber\\
& = \frac{1}{n^m} \sum_{\tilde{\vect \imath}_{m+1}}\sigma^2(\tilde{\vect \imath}_{m+1})\left[ \frac{1}{n} \sum_{i_q = 1}^n \sigma_{i_p i_q}^2 - 1 \right ] \label{eq:main term_1}\\
&+ \frac{1}{n^m} \sum_{\tilde{\vect \imath}_{m+1}}\sigma^2(\tilde{\vect \imath}_{m+1})\label{eq:main term_2}\\
&+ o_m(1),\nonumber
\end{align}
where we have added some graphs from category $C_3$ and use the similar bounds as for $S_3$ term. Now we will show that the term~\eqref{eq:main term_1} is of order $o_m(1)$. Note that
\begin{align} \label{eq: first term}
&\frac{1}{n^m} \sum_{\tilde{\vect \imath}_{m+1}}\sigma^2(\tilde{\vect \imath}_{m+1})\left | \frac{1}{n} \sum_{i_q = 1}^n \sigma_{i_p i_q}^2 - 1 \right |  \\
&=\frac{1}{n} \sum_{i_p = 1}^n \left | \frac{1}{n} \sum_{i_q = 1}^ n \sigma_{i_p i_q}^2 - 1 \right | \frac{1}{n^{m-1}}\sum_{\hat{\vect \imath}_{m+1}} \sigma^2(\tilde{\vect \imath}_{m+1}) \nonumber.
\end{align}
We can sequentially delete leafs from the tree and using~\eqref{eq: variance prop 3} write the bound
\begin{equation} \label{eq:first term 2}
\frac{1}{n^{m-1}} \sum_{\hat{\vect \imath}_{m+1}}\sigma^2(\tilde{\vect \imath}_{m+1}) \le C^{2(m-1)}.
\end{equation}
By~\eqref{eq:first term 2} and~\eqref{eq: variance prop 2} we have shown that~\eqref{eq:main term_1} is of order $o_m(1)$. For the second term~\eqref{eq:main term_2}
we can repeat the above procedure and stop after $m-1$ steps when we arrive at only two vertices in the tree. In the last step we can use the result~\eqref{eq: variance first step}.
Finally we get
$$
S_1 = \frac{1}{n^{m+1}} \sum_{\Gamma_1(2m)}\sum_{\vect i_{m+1}} \sigma^2(\vect i_{m+1}) = \frac{1}{m+1}\binom{2m}m + o_m(1),
$$
which proves Theorem~\ref{th: gaussian case}.


\bibliographystyle{plain}
\bibliography{literatur}
\end{document}